\documentclass[11pt]{amsart}
\usepackage[textheight=8.55in, textwidth=6.5in]{geometry}

\usepackage{amscd, amsmath, amssymb, amsthm, eurosym, mathdots, mathtools}

\usepackage{color}
\usepackage{graphicx}
\usepackage{xcolor}

\usepackage{tikz}
\usepackage{tikz-cd}
\usetikzlibrary{arrows, cd}
\usetikzlibrary{babel}
\usetikzlibrary{datavisualization.formats.functions}
\usetikzlibrary{decorations.markings}

\usepackage{caption, float, multicol, tabu}
\usepackage[autostyle, english = american]{csquotes}
\MakeOuterQuote{"}

\usepackage{fancyhdr}
\usepackage{hyperref}
\usepackage[final]{pdfpages}

\usepackage[]{breqn}
\usepackage{enumerate}

\usepackage{bm}
\usepackage{bbm}
\usepackage{comment}
\usepackage{physics}
\usepackage{verbatim}
\usepackage[all]{xy}

\newtheorem{theorem}{Theorem}[section]
\newtheorem*{theorem*}{Theorem}

\newtheorem*{proposition*}{Proposition}
\newtheorem{lemma}[theorem]{Lemma}
\newtheorem*{lemma*}{Lemma}

\newtheorem*{corollary*}{Corollary}

\newtheorem*{conjecture*}{Conjecture}

\newtheorem*{definition*}{Definition}

\theoremstyle{remark}

\newtheorem*{remark*}{Remark}

\newtheorem*{example*}{Example}

\numberwithin{equation}{section}

\newcommand{\C}{\mathbb{C}}

\renewcommand{\H}{\mathbb{H}}

\newcommand{\N}{\mathbb{N}}

\newcommand{\Q}{\mathbb{Q}}

\newcommand{\Z}{\mathbb{Z}}

\newcommand{\cU}{\mathcal{U}}

\title{MacMahonesque partition functions detect sets related to primes}
\author{Kevin Gomez}
\address{Dept. of Mathematics, University of Virginia, Charlottesville, VA 22904}
\email{vhe4ht@virginia.edu}


\begin{document}

\begin{abstract}
    Recent work by Craig, van Ittersum, and Ono constructs explicit expressions in the partition functions of MacMahon that detect the prime numbers. Furthermore, they define generalizations, the MacMahonesque functions, and prove there are infinitely many such expressions in these functions. Here, we show how to modify and adapt their construction to detect cubes of primes as well as primes in arithmetic progressions.
\end{abstract}

\maketitle

\section{Introduction and Statement of Results}

Recent work has investigated equations in partition functions which detect the set of primes. To be precise, in \cite{CraigIttersumOno}, Craig, van Ittersum, and Ono construct prime-detecting expressions using MacMahon's~\cite{MacMahon} $q$-series
\begin{align} \label{def:Ua} 
    \mathcal{U}_a(q)=\sum_{n\geq1}M_a(n)\,q^n:=\sum_{0< s_1<s_2<\cdots<s_a} \frac{q^{s_1+s_2+\cdots+s_a}}{(1-q^{s_1})^2(1-q^{s_2})^2\cdots(1-q^{s_a})^2}.
\end{align}
For $a \geq 1$, $M_a(n)$ sums the products of the multiplicities of partitions of $n$ with $a$ different part sizes. That is, we have
\begin{equation}\label{def:Ma}
    M_a(n) = \sum_{\substack{0 < s_1 < s_2 < \dots < s_a \\ n = m_1s_1 + m_2s_2 + \dots+ m_as_a}} m_1m_2 \cdots m_a.
\end{equation}
For positive integers $n$, they prove in Theorem 1 of \cite{CraigIttersumOno} that
\begin{align*}
    (n^2-3n+2) M_1(n) - 8M_2(n)\geq 0,
\end{align*}
and for $n \geq 2$, they prove that this expression vanishes if and only if $n$ is prime. 

They also find a number of other examples, and conjecture that there are only finitely many such intrinsically independent expressions in MacMahon's functions that detect the primes. However, they obtain a natural generalization which yields infinitely many such prime detectors. Namely, for a positive integer $a$ and a vector $\vec{a} = (v_1,v_2,\dots,v_a) \in \N^a$, they define the \textit{MacMahonesque partition function}
\begin{align*}
	M_{\vec{a}}(n) := \sum_{\substack{0 < s_1 < s_2 < \dots < s_a \\ n = m_1s_1 + m_2s_2 + \dots + m_as_a}} m_1^{v_1} m_2^{v_2} \cdots m_a^{v_a},
\end{align*}
which sums monomials of degree $\abs{\vec{a}} := v_1 + v_2 + \dots + v_a$ in the part multiplicities of partitions of $n$. Note that $M_a(n) = M_{\vec{a}}(n)$ for $\vec{a} = (1, \dots, 1) \in \N^a$. They find, for example, that for positive integers $n$, we have
\begin{displaymath}
\begin{split}
    \Psi_1(n):=63M_{(2,2)}(n)-12&M_{(3,0)}(n) 
    -39M_{(3,1)}(n)-12M_{(1,3)}(n)\\
    &+80M_{(1,1,1)}(n)-12M_{(2,0,1)}(n)+12 M_{(2,1,0)}(n) +12 M_{(3,0,0)}(n)\geq 0,
\end{split}
\end{displaymath}
and for $n \geq 2$ that $\Psi_1(n)$ likewise vanishes if and only if $n$ is prime.

These works follow a general philosophy of Schneider, where classical number theoretic topics are informed by partition theory (for example, see~\cite{SchneiderDawseyJust, Schneider1, Schneider2, Schneider3}). In this vein, it is natural to ask if this theorem generalizes to detect other sets. To be precise, which subsets of primes or sets related to primes are exactly detectable from these functions without any further generalization?

To answer this question, we specify two related notion of detecting a set by a $q$-series. Given a set $S \subseteq \N$, we say that a $q$-series $\sum_{n \geq 0} a_n q^n$ \textit{detects} $S$ if, for $n \geq 2$, $a_n = 0$ if and only if $n \in S$. A $q$-series furthermore \textit{strongly detects} $S$ if $a_n \geq 0$ for $n \geq 1$.

We find that cubes of primes are detectable using the MacMahonesque functions in infinitely many ways. Define the generating series
\begin{align*}
    \cU_{\vec{a}}(n) := \sum_{n \geq 1} M_{\vec{a}}(n)q^n.
\end{align*}
We then have the following theorem.

\begin{theorem} \label{thm:primes-cubed}
    There exist infinitely many linearly independent expressions
    \begin{align*}
        \sum_{\abs{\vec{a}} \leq d} c_{\vec{a}}\cU_{\vec{a}}(q)
    \end{align*}
    which detect the set $S$ of primes cubed, where $c_{\vec{a}} \in \Z$ and $d \geq 21$.
\end{theorem}

\begin{remark*}
    It is natural to ask about other prime powers. Firstly, we consider only odd powers, as our method makes use of quasimodular forms of level 1. The method of proof of Theorem \ref{thm:primes-cubed} then breaks down for powers 5 or higher, because the inequalities required do not continue to hold.
\end{remark*}

\begin{example*}
    For integers $n \geq 2$, we have that
    \begin{align*}
        g^{\ast}(n) &= \overbrace{1118 \cdots 4000}^{63 \text{ digits}} M_{(1)}(n) - \overbrace{1667 \cdots 1760}^{63 \text{ digits}} M_{(3)}(n) + \cdots + \overbrace{2160 \cdots 7200}^{55 \text{ digits}} M_{(15,1,1)}(n)
    \end{align*}
    vanishes if and only if $n = p^3$ for some prime $p$. A complete expression is given in the Appendix.
\end{example*}

We also consider twists of $\cU_{\vec{a}}(q)$ by a root of unity, which enables us to strongly detect primes in arithmetic progressions. In particular, we have the following theorem.
\begin{theorem} \label{thm:primes-ap}
    Let $t$ be a fixed positive integer, $0\leq r< t$ be coprime to $t$, and $\zeta_t$ be a primitive $t^{\text{th}}$ root of unity. Then there exist infinitely many linearly independent expressions
    \begin{align*}
        \sum_{\abs{\vec{a}} \leq d} \sum_{s=0}^{t-1} c_{\vec{a}}(s)\cU_{\vec{a}}(\zeta_t^sq)
    \end{align*}
    which strongly detect the set $S$ of primes congruent to $r$ modulo $t$, where $c_{\vec{a}}(s) \in \C$ and $d \geq 7$.
\end{theorem}

\begin{remark*}
    Craig, in \cite{Craig}, constructs series as in Theorem \ref{thm:primes-ap} which detect primes in arithmetic progression. Specifically, he obtains quasimodular forms with level which detect primes in certain pairs of arithmetic progressions, as a consequence of his primary theorems concerning $q$-multiple zeta values at level $N$.
\end{remark*}

\begin{example*}
    For integers $n \geq 2$, we have that
    \begin{align*}
        f^\ast(n) := 4(199 + 21\omega^{n-1} + 21\omega^{2n-2})M_{(1)}(n) + \cdots + 120960M_{(3,1,1)}(n) - 161280M_{(1,1,1,1)}(n)
    \end{align*}
    vanishes if and only if $n$ is a prime congruent to $1$ modulo $3$, where $\omega := e^{2\pi i/3}$ is a primitive third root of unity. A complete expression is given in the Appendix.
\end{example*}

\subsection{Acknowledgements}

The author would like to thank Will Craig, Ken Ono, Ajit Singh, and the anonymous referee for very helpful comments in the writing of this manuscript. The author also thanks the Thomas Jefferson Fund and grants from the NSF (DMS-2002265 and DMS-2055118).

\section{Quasimodular Forms}

We require the basic elements of the theory of quasimodular forms; see \cite[Section 5.3]{Zagier} for an exposition. For even integers $k \geq 1$, we consider the $k^{\text{th}}$ \textit{Bernoulli numbers} $B_k$ and the \textit{Eisenstein series}
\begin{align} \label{def:g}
    G_k(\tau) := -\frac{B_k}{2k} + \sum_{n=1}^{\infty} \sigma_{k-1}(n)q^n,
\end{align}
with $q := e^{2\pi i\tau}, \tau\in\H$ and $\sigma_{k-1}(n) := \sum_{d \mid n} d^{k-1}$. It is a well-known fact (for example, see~\cite{Ono}) that $G_k$ is quasimodular, and furthermore modular if $k\geq 4$. We denote the algebra of quasimodular forms by $\widetilde{M}$, the subspace of forms with weight $k$ by $\widetilde{M}_k$, and the space of forms with mixed weight $\leq k$ by $\widetilde{M}_{\leq k}$.

We also require the $q$-differential operator
\begin{align}
    D := \frac{1}{2\pi i} \dfrac{d}{d\tau} = q \dfrac{d}{dq}.
\end{align}
For $k \geq 0$, we observe that $D^k : q^n \mapsto n^k q^n$. Ramanujan famously proved derivative identities for $G_2$, $G_4$, and $G_6$  (for example, see \cite[Section 7]{Ramanujan}):
\begin{align} \label{eq:ramanujan}
    DG_2 = - 2G_2^2 + \dfrac{5}{6} G_4, \ \ \ D G_4 = -8 G_2 G_4 + \dfrac{7}{10} G_6, \ \ \ DG_6 = -12 G_2 G_6 + \dfrac{400}{7} G_4^2.
\end{align}
Since $\widetilde{M}$ is generated by $G_2$, $G_4$, and $G_6$, these identities show that $D$ defines a map on the algebra of quasimodular forms which increases weights by $2$; that is, we have that
$D~\colon \widetilde{M}_{k} \to \widetilde{M}_{k+2}.$

\section{Detecting cubes of primes}
Utilizing elementary properties of $D$, we produce quasimodular forms which detect cubes of primes.

\begin{theorem} \label{thm:help}
    Let $k,\ell$ be non-negative odd integers with $\ell > k$. For all $n \geq 2$, the $n^{\text{th}}$ Fourier coefficient of
    \begin{align} \label{def:gkl}
        g_{k,\ell} := (D^{3\ell} + D^{2\ell} + D^{\ell} + 1)G_{3k+1} - (D^{3k} + D^{2k} + D^k + 1)G_{3\ell + 1}
    \end{align}
    vanishes if and only if $n = p^3$ for some prime $p$.
\end{theorem}

We require a number of lemmas to establish this result. We define, for $d \mid n$,
\begin{align} \label{def:aklnd}
    a_{k,\ell}(n,d) := (n^{3\ell} + n^{2\ell} + n^{\ell} + 1) \times d^{3k} - (n^{3k} + n^{2k} + n^{k} + 1) \times d^{3\ell}.
\end{align}

We let $a_{k,\ell}(n)$ be the $n$th Fourier coefficient of $g_{k,\ell}$, which is given by
\begin{align} \label{def:aklnn}
    a_{k,\ell}(n) = \sum_{d \mid n} a_{k,\ell}(n,d) = (n^{3\ell} + n^{2\ell} + n^{\ell} + 1) \times \sigma_{3k}(n) - (n^{3k} + n^{2k} + n^{k} + 1) \times \sigma_{3\ell}(n).
\end{align}

\pagebreak
\begin{lemma} \label{lem:asign}
    For all $n \geq 2$ and $d \mid n$, we have that $a_{k,\ell}(n,d) < 0$ if and only if $d = n$.
\end{lemma}

\begin{proof}
    We first simplify and extract powers of $n$ to obtain
    \begin{align*}
        a_{k,\ell}(n,n) &= (n^{3\ell} + n^{2\ell} + n^{\ell} + 1) \times n^{3k} - (n^{3k} + n^{2k} + n^{k} + 1) \times n^{3\ell} \\
        &= (n^{2\ell} + n^{\ell} + 1) \times n^{3k} - (n^{2k} + n^{k} + 1) \times n^{3\ell} \\
        &= n^{3(k+\ell)} (n^{-\ell} - n^{-k} + n^{-2\ell} - n^{-2k} + n^{-3\ell} - n^{-3k}).
    \end{align*}
    Since $\ell > k$, $n^{-j\ell} - n^{-jk} < 0$ for $j > 0$. Thus, $a_{k,\ell}(n,n) < 0$.

    Conversely, if $d \leq n/2$, we may bound $a_{k,\ell}(n,d)$ from below and simplify to obtain
    \begin{align*}
        a_{k,\ell}(n,d) &\geq d^{3k} \left((n^{3\ell} + n^{2\ell} + n^{\ell} + 1) - (n^{3k} + n^{2k} + n^{k} + 1) \times (n/2)^{3(\ell - k)} \right) \\
        &= d^{3k}n^{3\ell}\left( \frac{n^{\ell} - n^{-3\ell}}{n^{\ell} - 1} - \frac{n^{k} - n^{-3k}}{n^k - 1} \times \frac{1}{2^{3(\ell - k)}} \right).
    \end{align*}
    We have, for all $n \geq 2$ and $m \geq 1$,
    \begin{align*}
        1 < \frac{n^m - n^{-3m}}{n^m - 1} < 2,
    \end{align*}
    whereby
    \begin{align*}
         d^{3k}n^{3\ell}\left( \frac{n^{\ell} - n^{-3\ell}}{n^{\ell} - 1} - \frac{n^{k} - n^{-3k}}{n^k - 1} \times \frac{1}{2^{3(\ell - k)}} \right) &\geq d^{3k}n^{3\ell} \left(1  - \frac{1}{2^{3(\ell - k) - 1}} \right) \geq 1.
    \end{align*}
\end{proof}

\begin{lemma} \label{lem:abound}
    For all $n \geq 2$ and $1 \leq d < n$, we have
    \begin{align*}
        a_{k,\ell}(n,d) < n^{3\ell}d^{3k}\left(1 + \frac{1}{n^{\ell} - 1}\right).
    \end{align*}
    Furthermore, for all $n \geq 2$,
    \begin{align*}
        \abs{a_{k,\ell}(n,n)} > n^{2k+3\ell}\left(1 - \frac{1}{n^2}\right).
    \end{align*}
\end{lemma}

\begin{proof}
    For $d < n$, we have the bound
    \begin{align*}
        a_{k,\ell}(n,d) &\leq (n^{3\ell} + n^{2\ell} + n^{\ell} + 1) \times d^{3k} - (n^{3k} + n^{2k} + n^{k} + 1) \times d^{3\ell} \\
        &\leq (n^{3\ell} + n^{2\ell} + n^{\ell} + 1) \times d^{3k} - (n^{3k} + n^{2k} + n^{k} + 1) \times d^{3k},
    \end{align*}
   where the second inequality follows since $\ell > k$. Extracting the common factor of $d^{3k}$, we have
    \begin{align*}
        d^{3k}(n^{3\ell} - n^{3k} + n^{2\ell} - n^{2k} + n^{\ell} - n^k) &\leq d^{3k}(n^{3\ell} + n^{2\ell} + n^{\ell}) = d^{3k}n^{3\ell}(1 + n^{-\ell} + n^{-2\ell}).
    \end{align*}
    Writing
    \begin{align*}
        1 + n^{-\ell} + n^{-2\ell} &= \frac{1 - n^{-3\ell}}{1 - n^{-\ell}} < \frac{1}{1 - n^{-\ell}} = \frac{n^{\ell}}{n^{\ell} - 1} = 1 + \frac{1}{n^{\ell} - 1},
    \end{align*}
    we obtain the first part of our result.

    For $d = n$, we find through straightforward algebraic manipulation that
    \begin{align*}
        \frac{1}{n^{3(k+\ell)}}\abs{a_{k,\ell}(n,n)} &= (n^{3k} + n^{2k} + n^{k} + 1) \times n^{-3k} - (n^{3\ell} + n^{2\ell} + n^{\ell} + 1) \times n^{-3\ell}  \\
        &= (n^{2k} + n^{k} + 1) \times n^{-3k} - (n^{2\ell} + n^{\ell} + 1) \times n^{-3\ell} \\
        &= \frac{n^{3k} - 1}{n^k - 1} \times n^{-3k} - \frac{n^{3\ell} - 1}{n^{\ell} - 1} \times n^{-3\ell},
    \end{align*}
    which we simplify further to
    \begin{align} \label{eq:aklnn1}
        \frac{1}{n^{3(k+\ell)}}\abs{a_{k,\ell}(n,n)} &= \frac{1 - n^{-3k}}{n^k - 1} - \frac{1 - n^{-3\ell }}{n^{\ell} - 1}.
    \end{align}
    Since $\ell \geq k + 2$, we may write
    \begin{align*}
        \frac{1 - n^{-3\ell }}{n^{\ell} - 1} < \frac{1}{n^{\ell} - 1} \leq \frac{1}{n^{k+2} - 1}.
    \end{align*}
    Inserting this bound into \eqref{eq:aklnn1}, we have
    \begin{align}  \label{eq:aklnn}
        \frac{1}{n^{3(k+\ell)}}\abs{a_{k,\ell}(n,n)} &> \frac{1 - n^{-3k}}{n^k - 1} - \frac{1}{n^{k+2} - 1} = \frac{1}{n^k - 1} - \frac{1}{n^{4k} - n^{3k}} - \frac{1}{n^{k+2} - 1}.
    \end{align}
    Since $n^{4k} - n^{3k} > n^{3k}$ for all $n \geq 2$, we have
    \begin{align} \label{eq:aklnn2}
        \frac{1}{n^k - 1} - \frac{1}{n^{4k} - n^{3k}} - \frac{1}{n^{k+2} - 1} > \frac{1}{n^k - 1} - \frac{1}{n^{3k}} - \frac{1}{n^{k+2} - 1}
    \end{align}
    Combining terms in the right-hand side of \eqref{eq:aklnn2}, we obtain
    \begin{align} \label{eq:aklnn3}
        \frac{1}{n^{3k}} + \frac{1}{n^{k+2} - 1} &= \frac{n^{3k} + n^{k+2} - 1}{n^{3k}(n^{k+2} - 1)} = \frac{1 + n^{-2k+2} - n^{-3k}}{n^{k+2} - 1}.
    \end{align}
    Meanwhile, we compute
    \begin{align} \label{eq:aklnn4}
        \frac{1}{n^k - 1} - \left(\frac{1}{n^k} - \frac{1}{n^{k+2}}\right) &= \frac{n^{k+2} - (n^k - 1)(n^2 - 1)}{n^{k+2}(n^k - 1)} = \frac{1 + n^{-k+2} - n^{-k}}{n^{k+2} - n^2}.
    \end{align}
    We now compare \eqref{eq:aklnn3} and \eqref{eq:aklnn4}, finding that
    \begin{align} \label{eq:aklnn5}
        \frac{1 + n^{-2k+2} - n^{-3k}}{n^{k+2} - 1} < \frac{1 + n^{-2k+2}}{n^{k+2} - 1} < \frac{1 + n^{-2k+2}}{n^{k+2} - n^2} < \frac{1 + n^{-k+2} - n^{-k}}{n^{k+2} - n^2},
    \end{align}
    where the rightmost inequality follows from $n^2 < n^{k+2} - n^k$ for all $n \geq 2$ and $k \geq 1$.

    Thus, applying \eqref{eq:aklnn2} and \eqref{eq:aklnn5} to \eqref{eq:aklnn}, we have
    \begin{align*}
        \frac{1}{n^{3(k+\ell)}}\abs{a_{k,\ell}(n,n)} &> \frac{1}{n^k - 1} - \frac{1}{n^{3k}} - \frac{1}{n^{k+2} - 1} \\
        &> \frac{1}{n^k - 1} - \left(\frac{1}{n^k - 1} - \left(\frac{1}{n^k} - \frac{1}{n^{k+2}}\right)\right) \\
        &= \frac{1}{n^k}\left(1 - \frac{1}{n^2}\right),
    \end{align*}
    yielding our result.
\end{proof}

\pagebreak
We now obtain a condition on $n$ which dictates the sign of $a_{k,\ell}(n)$.

\begin{lemma} \label{lem:pos}
    If $n > p^3$ for some $p \mid n$ prime, then $a_{k,\ell}(n) > 0$.
\end{lemma}

\begin{proof}
    Let $p \mid n$ be such that $n > p^3$. We evaluate via \eqref{def:aklnd}
    \begin{align} \label{eq:posrat}
        a_{k,\ell}(n,n/p) + a_{k,\ell}(n,n) &= \frac{(n^{4\ell} - 1)(p^{3k} + 1)}{n^{\ell} - 1} \times \frac{n^{3k}}{p^{3k}} - \frac{(n^{4k} - 1)(p^{3\ell} + 1)}{n^k - 1} \times \frac{n^{3\ell}}{p^{3\ell}}.
    \end{align}
    We now bound the ratio of the terms in the right-hand side of \eqref{eq:posrat}. Since $\ell > k$, we have $n^{4\ell} - 1 \geq n^{4\ell} - n^{4(\ell -k)}$ and thus $(n^{4\ell} - 1)/(n^{4k} - 1) \geq n^{4(\ell - k)}$. Hence, we find through straightforward algebraic manipulation that
    \begin{align} \label{eq:h13}
        \frac{p^{3(\ell - k)}(n^{4\ell} - 1)(p^{3k} + 1) (n^k - 1)}{n^{3(\ell - k)}(n^{4k} - 1)(p^{3\ell} + 1)(n^{\ell} - 1)} &\geq (np^3)^{\ell - k} \times \frac{(p^{3k} + 1) (n^k - 1)}{(p^{3\ell} + 1)(n^{\ell} - 1)}.
    \end{align}
    Distributing the factors in $(np^3)^{\ell - k}$, we obtain
    \begin{align} \label{eq:h12}
        (np^3)^{\ell - k} \times \frac{(p^{3k} + 1) (n^k - 1)}{(p^{3\ell} + 1)(n^{\ell} - 1)} &= \frac{p^{3\ell} + p^{3(\ell - k)}}{p^{3\ell} + 1} \times \frac{n^k - 1}{n^k - n^{k - \ell}} > \frac{p^{3\ell} + p^{3(\ell - k)}}{p^{3\ell} + 1} \times \frac{n^k - 1}{n^k}.
    \end{align}
    
    Now, let
    \begin{align*}
        h_1(x) := \frac{p^{3\ell} + p^{3(\ell - k)}}{p^{3\ell} + 1} \times \frac{x^k - 1}{x^k}, \text{ for } x \geq 1,
    \end{align*}
    so that $h_1(n)$ is a lower bound for the left hand side of \eqref{eq:h13} thanks to \eqref{eq:h12}. It is clear that $h_1(x)$ is an increasing function, and thus it suffices to show that $h_1(p^3 + p) > 1$ to deduce that $h_1(n) > 1$. We thus calculate and expand
    \begin{align*}
        h_1(p^3 + p) &= \frac{p^{3\ell} + p^{3(\ell - k)}}{p^{3\ell} + 1} \times \frac{(p^3 + p)^{k} - 1}{(p^3 + p)^{k}} \\
        &= \left( 1 + \frac{p^{3(\ell - k)} - 1}{p^{3\ell} + 1} \right)\left(1 - \frac{1}{(p^3 + p)^{k}}\right) \\
        &= 1 - \frac{1}{(p^3 + p)^{k}} + \frac{p^{3(\ell-k)} - 1}{p^{3\ell} + 1} - \frac{p^{3(\ell-k)}-1}{(p^{3\ell} + 1)(p^3 + p)^{k}}.
    \end{align*}
    We now group the terms above to obtain
    \begin{align} \label{eq:h14}
        h_1(p^3 + p) &= 1 + \frac{p^{3(\ell - k)}}{(p^{3\ell} + 1)(p^3 + p)^k}\left( (p^3 + p)^k\left(1 - p^{3(k-\ell)}\right) - p^{3k} - 1 \right)
    \end{align}
    and wish to show that this value is positive. We first have by straightforward algebraic manipulation that
    \begin{align*}
        (p^3 + p)^k\left(1 - p^{3(k-\ell)}\right) - p^{3k} - 1 &= ((p^3 + p)^k - p^{3k})\left(1 - p^{3(k-\ell)}\right) - p^{3(2k - \ell)} - 1.
    \end{align*}
    Thanks to the binomial theorem,
    \begin{align*}
        (p^3 + p)^k - p^{3k} > kp^{3(k - 1) + 1},
    \end{align*}
    whereby
    \begin{align*}
         ((p^3 + p)^k - p^{3k})\left(1 - p^{3(k-\ell)}\right) - p^{3(2k - \ell)} - 1 &> k\left(1 - p^{3(k - \ell)} \right)p^{3(k - 1) + 1} - p^{3(2k - \ell)} - 1.
    \end{align*}
    This equates to
    \begin{align} \label{eq:h15}
         k\left(1 - p^{3(k - \ell)} \right)p^{3(k - 1) + 1} - p^{3(2k - \ell)} - 1 &= p^{3(2k - \ell)}\left(k\left(p^{3(\ell-k-1)+1} - p^{-2}\right) - 1\right) - 1.
    \end{align}
    We then observe that
    \begin{align*}
        p^{3(\ell-k-1)+1} - p^{-2} - k^{-1} > p^{3(\ell - 2k)}
    \end{align*}
    since $k \geq 1$ and $p \geq 2$. This grants
    \begin{align*}
        p^{3(2k - \ell)}\left(k\left(p^{3(\ell-k-1)+1} - p^{-2}\right) - 1\right) > 1,
    \end{align*}
    which combined with \eqref{eq:h14} and \eqref{eq:h15} completes the proof that $h_1(n) > 1$. Since $a_{k,\ell}(n) \geq a_{k,\ell}(n,n/p) + a_{k,\ell}(n,n)$ by Lemma \ref{lem:asign}, we obtain our result.
\end{proof}

\begin{lemma} \label{lem:zero}
    If $n = p^3$ for some prime $p$, then $a_{k,\ell}(n) = 0$.
\end{lemma}

\begin{proof}
    For $s > 0$, we have
    \begin{align*}
        \sigma_s(p^3) &= p^{3s} + p^{2s} + p^s + 1,
    \end{align*}
    which substituting into \eqref{def:aklnn} grants
    \begin{align*}
        a_{k,\ell}(p^3) &= (p^{9\ell} + p^{6\ell} + p^{3\ell} + 1)(p^{9k} + p^{6k} + p^{3k} + 1) - (p^{9k} + p^{6k} + p^{3k} + 1)(p^{9\ell} + p^{6\ell} + p^{3\ell} + 1) = 0.
    \end{align*}
\end{proof}

\begin{lemma} \label{lem:neg}
    If $n < p^3$ for all $p \mid n$, then $a_{k,\ell}(n) < 0$.
\end{lemma}

\begin{proof}
    Since $n \leq p^3 - p$ for $p$ the smallest prime dividing $n$, we have three cases. \\

    \noindent\textbf{Case 1:} $n = p$. We first expand
    \begin{align*}
        a_{k,\ell}(p) &= (p^{3\ell} + p^{2\ell} + p^{\ell} + 1)(p^{3k} + 1) - (p^{3k} + p^{2k} + p^k + 1)(p^{3\ell} + 1) \\
        &= p^{3k + 2\ell} + p^{3k+\ell} + p^{2\ell} + p^{\ell} - p^{3\ell+2k} - p^{3\ell+k} - p^{2k} - p^k.
    \end{align*}
    Extracting the largest powers of $p$, we find that
    \begin{align*}
        a_{k,\ell}(p) &= p^{2k+2\ell}(p^{k} + p^{-2k} + p^{-\ell-2k} - p^{\ell} - p^{-2\ell} - p^{-k-2\ell})).
    \end{align*}
    All negative powers of $p$ are bounded above by $1$, so we obtain,
    \begin{align*}
        a_{k,\ell}(p) &\leq p^{2k+2\ell}(p^{k} - p^{\ell} + p^{-2k} + p^{-\ell-2k}) \leq p^{2k+2\ell}(p^k - p^{\ell} + 2) \leq p^{2k+2\ell}(2^1 - 2^3 + 2) < 0.
    \end{align*}

    \noindent\textbf{Case 2:} $n = p^2$. As in the previous case, we expand out $a_{k,\ell}(p^2)$ and collect the largest powers of $p$. Discarding negative addends and bounding negative powers of $p$ by $1$, we obtain
    \begin{align*}
        a_{k,\ell}(p^2) &= (p^{6\ell} + p^{4\ell} + p^{2\ell} + 1)(p^{6k} + p^{3k} + 1) - (p^{6k} + p^{4k} + p^{2k} + 1)(p^{6\ell} + p^{3\ell} + 1) \\
        &\leq p^{4k+4\ell}(p^{2k} - p^{2\ell} + 10) \\
        &\leq p^{4k+4\ell}(2^2 - 2^6 + 10) \\
        &< 0.
    \end{align*}

    \noindent\textbf{Case 3:} $n = pp'$ with $p' > p$ prime. Let
    \begin{align*}
        h_2(x) := (1 + p^{3k} + (x/p)^{3k})\left(1 + \frac{1}{x^k}\right) - x^{2k}\left(1 - \frac{1}{p^4}\right).
    \end{align*}
    We first write $a_{k,\ell}(n)$ via \eqref{def:aklnn} as
    \begin{align*}
        a_{k,\ell}(n) &= a_{k,\ell}(n,1) + a_{k,\ell}(n,p) + a_{k,\ell}(n,n/p) + a_{k,\ell}(n,n).
    \end{align*}
    We then apply Lemma \ref{lem:abound} to find that
    \begin{align*}
        a_{k,\ell}(n) &< n^{3\ell}\left[(1 + p^{3k} + (n/p)^{3k})\left(1 + \frac{1}{n^{\ell} - 1}\right) - n^{2k}\left(1 - \frac{1}{n^2}\right)\right] \\
        &< n^{3\ell}\left[(1 + p^{3k} + (n/p)^{3k})\left(1 + \frac{1}{n^k}\right) - n^{2k}\left(1 - \frac{1}{n^2}\right)\right].
    \end{align*}
    Since $p < \sqrt{n}$, $p^4 < n^2$, and hence we obtain
    \begin{align*}
        a_{k,\ell}(n) &< n^{3\ell}\left[(1 + p^{3k} + (n/p)^{3k})\left(1 + \frac{1}{n^k}\right) - n^{2k}\left(1 - \frac{1}{p^4}\right)\right] = n^{3\ell}h_2(n).
    \end{align*}
    We then compute
    \begin{align*}
        h_2'(x) &= -\frac{kx^{-k-1}}{p^4}\left[p^{3k} + 1 - 3p^{-3k}x^{4k} - (2p^{-3k} - 2p^{-4})x^{3k}\right]
    \end{align*}
    and wish to analyze its sign. We see that
    \begin{align*}
        3p^{-3k}n^{4k} + (2p^{-3k} - 2p^{-4})n^{3k} - (p^{3k} + 1) = 0
    \end{align*}
    has exactly one positive real root $w$ by Descartes' rule of signs. Substituting $n = p^2$, we have
    \begin{align*}
        3p^{-3k}p^{8k} + (2p^{-3k} - 2p^{-4})p^{6k-4} - (p^{3k} + 1) &= -2p^{6k} + 3p^{5k} + 2p^{3k} - p^{3k} - 1 < 0
    \end{align*}
    for all $p \geq 2$. Likewise, for $n = p^3$, we have
    \begin{align*}
        3p^{-3k}p^{12k} + (2p^{-3k} - 2p^{-4})p^{9k} - (p^{3k} + 1) = 3p^{9k} - 2p^{9k-4} + 2p^{6k} - p^{3k} - 1 > 0
    \end{align*}
    for all $p \geq 2$. We thus deduce that $w \in (p^2,p^3)$, with $h_2(n)$ decreasing on $[p^2,w)$ and increasing on $(w,p^3]$. It is thus sufficient to verify that $h_2(n) < 0$ at $n = p^2$ and $n = p^3 - p$ to conclude this case.
    
    Indeed,
    \begin{align} \label{eq:h2p2}
        h_2(p^2) &= (1 + 2p^{3k})(1 + p^{-2k}) - p^{4k}(1 - p^{-4}) = -p^{4k} + p^{4k-2} + 2p^{3k} + 2p^k + 1 + p^{-2k} < 0
    \end{align}
    whenever $p \geq 3$ or $k \geq 3$. Furthermore, we have by straightforward algebraic manipulation
    \begin{align*}
        h_2(p^3 - p) &= (1 + p^{3k} + (p^2 - 1)^{3k})(1 + (p^3 - p)^{-k}) - (p^3 - p)^{2k}(1 - p^{-4}) \\
        &= -(p^3 - p)^{2k}(1 - p^{-4}) + (p^2 - 1)^{3k} + p^{3k} + (p^{3k} + 1)(p^3 - p)^{-k} + 1 + p^{-3k}(p^3 - p)^{2k},
    \end{align*}
    which we may bound simply by
    \begin{align*}
        h_2(p^3 - p) &< -(p^3 - p)^{2k}(1 - p^{-4}) + 4p^{3k} = (p^2 - 1)^{2k}(-p^{2k}(1 - p^{-4}) +5(p^2 - 1)^k) + 4p^{3k}
    \end{align*}
    since $p^{3k}$ is the largest remaining positive term. By the binomial theorem, $m^{2k} - (m^2 + 1)^k < -km^{2k-2}$, whereby
    \begin{align} \label{eq:h2p3}
        (p^2 - 1)^{2k}(-p^{2k}(1 - p^{-4}) + (p^2 - 1)^k) + 4p^{3k} < (p^2 - 1)^{2k}(-k(p^2 - 1)^{k-1} + p^{2k-4}) + 4p^{3k}.
    \end{align}
    For $k \geq 3$, we have that $(p^2 - 1)^{2k} > 4p^{3k}$ whenever $p \geq 3$ or $k \geq 13$. Since $-k(p^2 - 1)^{k-1} + p^{-2k - 4} < -1$ (again by the binomial theorem), we obtain from \eqref{eq:h2p3} that $h_2(p^3 - p) < 0$. For $k = 1$, we have
    \begin{align*}
        (p^2 - 1)^2(-1 + p^{-2}) + 4p^3 = -p^{-2}(p^2 - 1)^3 + 4p^{3} < 0
    \end{align*}
    for all $p \geq 5$. Along with \eqref{eq:h2p2}, this resolves this case for all pairs $(k,p)$ except $(k,2)$ for $1 \leq k \leq 13$.

\pagebreak
    For these remaining pairs, we have that $p = 2$ implies $p' = 3$ so that $pp' = n < p^3$. Thus, we simply compute via \eqref{def:aklnn}
    \begin{align*}
        a_{1,\ell}(6) &= 252(6^{3\ell} + 6^{2\ell} + 6^{\ell} + 1) - 259(6^{3\ell} + 3^{3\ell} + 2^{3\ell} + 1).
    \end{align*}
    We apply straightforward bounds to obtain
    \begin{align*}
        a_{1,\ell}(6) &< 252(6^{3\ell} + 6^{2\ell} + 6^{\ell} + 1) - 259 \times 6^{3\ell} < -7 \times 6^{3\ell} + 252 \times 3 \times 6^{2\ell},
    \end{align*}
    which we may readily verify is negative for $\ell \geq 3$. This completes this case and the proof of the lemma.
\end{proof}

To prove Theorem \ref{thm:primes-cubed}, we require infinitely many such $q$-series which detect the set of primes cubed.
\begin{lemma} \label{lem:inf-cubed}
    For any fixed $k \geq 1$ odd, the elements of the set $\{g_{k,\ell}~\colon \ell > k \text{ odd}\}$ are linearly independent over $\C$.
\end{lemma}

\begin{proof}
    Let $k,\ell_1,\dots,\ell_j \geq 1$ be odd with $\ell_j > \cdots > \ell_1 > k$. Since $D$ raises the weight of a quasimodular form by $2$, we have that $D^{3\ell_i} G_{3k+1} \in \widetilde{M}_{3k+6\ell_i+1}$. Since all other addends in \eqref{def:gkl} have lower weight, any nontrivial linear combination $m_1g_{k,\ell_1} + \cdots + m_jg_{k,\ell_j}$ has a component of weight $3k + 6\ell_j + 1$, and hence is nonzero.
\end{proof}

\subsection{Proof of Theorem \ref{thm:help}}

Observe that for all $n \geq 2$, we have exactly three cases: $n > p^3$ for some $p \mid n$, $n < p^3$ for all $p \mid n$, and $n = p^3$. In the first two cases, $a_{k,\ell}(n) \neq 0$ by Lemmas \ref{lem:pos} and \ref{lem:neg} respectively, and in the third we have $a_{k,\ell}(n) = 0$ by Lemma \ref{lem:zero}.

\subsection{Proof of Theorem \ref{thm:primes-cubed}}

Let $\ell \geq 3$ be odd and $d = 6\ell + 3$. By Theorem \ref{thm:help}, $g_{1,\ell}$ detects the set of cubes of primes. The set $\{g_{1,\ell}~\colon \ell \geq 3 \text{ odd}\}$ thus contains infinitely many linearly independent $q$-series which detect prime cubes by Lemma \ref{lem:inf-cubed}. Since $D$ raises the weight of a quasimodular form by $2$, we have that $g_{1,\ell} \in \widetilde{M}_{\leq 6\ell+4}$, so by Theorem 19 (2) of \cite{CraigIttersumOno} we may write these series as
\begin{align*}
    g_{1,\ell} = \sum_{\abs{\vec{a}} \leq d} c_{\vec{a}}~\cU_{\vec{a}}(q)
\end{align*}
for constants $c_{\vec{a}} \in \Q$. This completes the proof of the theorem.

\section{Detecting primes in arithmetic progression}
Throughout, we fix positive integers $t \geq 2$ and $0 < r < t$ coprime to $t$, and let $\zeta_t$ be a primitive $t^{\text{th}}$ root of unity. In a similar fashion as the previous section, we utilize sums of Eisenstein series with restricted support to produce quasimodular forms which strongly detect the primes congruent to $r$ modulo $t$.

We define the series
\begin{align} \label{def:grt}
    G_{k}^{r,t}(\tau) := \frac{1}{t}\sum_{s=0}^{t-1} \zeta_t^{-rs}G_k(\tau + s/t).
\end{align}
It is clear from the definition that $G_k^{r,t}$ is a quasimodular form of weight $k$ and level $t^2$. 

\pagebreak
\begin{lemma} \label{lem:grt}
    We have that $q^{-r}G_k^{r,t}(\tau)$ is supported on powers of $q^t$. Moreover, we have
    \begin{align*}
        G_{k}^{r,t}(\tau) = \sum_{m=0}^{\infty} \sigma_{k-1}(mt + r)q^{mt+r}.
    \end{align*}
\end{lemma}

\begin{proof}
    We have by \eqref{def:g} and \eqref{def:grt}
    \begin{align} \label{eq:grt1}
        G_k^{r,t}(\tau) &= \frac{1}{t} \sum_{s=0}^{t-1} \left( - \frac{B_k\zeta_t^{-rs}}{2k} + \sum_{n=1}^{\infty} \zeta_t^{(n-r)s}\sigma_{k-1}(n)q^n\right).
    \end{align}
    It is well known that
    \begin{align*}
        \sum_{s=0}^{t-1} \zeta_t^{(n - r)s} =
        \begin{cases}
            t & t \mid n - r, \\
            0 & t \nmid n - r.
        \end{cases}
    \end{align*}
    Since $r \not\equiv 0 \pmod{t}$, the constant term of $G_k^{r,t}$ vanishes. We may then simplify \eqref{eq:grt1} to obtain
    \begin{align*}
        G_k^{r,t}(\tau) &= \sum_{\substack{n \geq 1 \\ n \equiv r \bmod{t}}} \sigma_{k-1}(n)q^n = \sum_{m=0}^{\infty} \sigma_{k-1}(mt + r)q^{mt + r}.
    \end{align*}
\end{proof}

We now establish that $G_k^{r,t}(\tau)$ can be used to construct infinitely many $q$-series which detect primes in arithmetic progression.
\begin{lemma} \label{lem:primes-ap}
    Let $k,\ell$ be non-negative odd integers with $\ell > k$. For all $n \geq 2$, the $n^{\text{th}}$ Fourier coefficient of
    \begin{align} \label{def:fklrt}
        f_{k,\ell}^{r,t} := (D^{\ell} + 1)G_{k+1} - (D^k + 1)G_{\ell + 1}^{r,t}
    \end{align}
    vanishes if and only if $n \equiv r \pmod{t}$ is prime. Furthermore, all of the coefficients of $f_{k,\ell}^{r,t}$ are non-negative.
\end{lemma}

\begin{proof}
    We see via \eqref{def:g} and \eqref{def:grt} that the $n^{\text{th}}$ Fourier coefficient of $f_{k,\ell}^{r,t}$ is given by
    \begin{align*}
        b_{k,\ell}^{r,t}(n) =
        \begin{cases}
            \sum_{d \mid n} ((n^{\ell} + 1)d^k - (n^k + 1)d^{\ell}) & n \equiv r \pmod{t}, \\
            \sum_{d \mid n} (n^{\ell} + 1)d^k & n \not\equiv r \pmod{t}.
        \end{cases}
    \end{align*}
    If $n \equiv r \pmod{t}$, then $b_{k,\ell}^{r,t}(n) \geq 0$ and vanishes if and only if $n$ is prime or $n = 1$ by Lemma 2.1 of \cite{CraigIttersumOno} (see also Lemma 2 of \cite{Lelievre}). It is then clear $b_{k,\ell}^{r,t}(n) > 0$ for all $n \not\equiv r \pmod{t}$.
\end{proof}

\begin{lemma} \label{lem:inf-ap}
    For any fixed $k \geq 1$ odd, the elements of the set $\{f_{k,\ell}^{r,t} ~\colon \ell > k \text{ odd}\}$ are linearly independent over $\C$.
\end{lemma}

\begin{proof}
    As in Lemma \ref{lem:inf-cubed}, we see that any nontrivial linear combination of $f_{k,\ell_1}^{r,t}, \dots, f_{k,\ell_j}^{r,t}$ for fixed $k$ and $\ell_j > \dots > \ell_1 > k$ odd has a component of weight $k + 2\ell_j + 1$ by \eqref{def:fklrt}, and is thus nonzero.
\end{proof}

\begin{lemma} \label{lem:main-ap}
    For all odd non-negative integers $k,\ell$ with $\ell > k$, there exist constants $c_{\vec{a}}(s) \in \C$ such that
    \begin{align*}
        f_{k,\ell}^{r,t} &= \sum_{\abs{\vec{a}} \leq k + 2\ell} \sum_{s=0}^{t-1} c_{\vec{a}}(s)\cU_{\vec{a}}(\zeta_t^sq).
    \end{align*}
\end{lemma}

\begin{proof}
    Since $D$ raises the weight of a quasimodular form by $2$, we have that
    \begin{align*}
        (D^{\ell} + 1)G_{k+1} \in \widetilde{M}_{\leq k+2\ell+1}, \ \ \ (D^k + 1)G_{\ell+1} \in \widetilde{M}_{\leq 2k+\ell+1}.
    \end{align*}
    Thus, by Theorem 19 (2) of \cite{CraigIttersumOno}, there exist constants $b_{\vec{a}}$, $b_{\vec{a}}' \in \Q$, such that
    \begin{align*}
        (D^{\ell} + 1)G_{k+1} = \sum_{\abs{\vec{a}} + \ell(\vec{a}) \leq k+2\ell + 1} b_{\vec{a}}~\cU_{\vec{a}}(q), \ \ \ (D^k + 1)G_{\ell+1} = \sum_{\abs{\vec{a}} + \ell(\vec{a}) \leq 2k+\ell + 1} b_{\vec{a}}'~\cU_{\vec{a}}(q).
    \end{align*}
    Then, by \eqref{def:grt}, we have
    \begin{align*}
        (D^k + 1)G_{\ell+1}^{r,t} &= \frac{1}{t} \sum_{s=0}^{t-1} \sum_{\abs{\vec{a}} + \ell(\vec{a}) \leq 2k+\ell + 1} \zeta_t^{-rs} b_{\vec{a}}'~\cU_{\vec{a}}(\zeta_t^s q),
    \end{align*}
    which we may rewrite with constants $b_{\vec{a}}'(s) := \frac{1}{t}\zeta_t^{-rs}b_{\vec{a}}'$ as
    \begin{align} \label{eq:grtexp}
        (D^k + 1)G_{\ell+1}^{r,t} &= \sum_{\abs{\vec{a}} + \ell(\vec{a}) \leq 2k+\ell + 1} \sum_{s=0}^{t-1} b_{\vec{a}}(s)\cU_{\vec{a}}(\zeta_t^sq).
    \end{align}
    Likewise, setting $b_{\vec{a}}(0) := b_{\vec{a}}$ and $b_{\vec{a}}(s) := 0$ for $s > 0$, we may write
    \begin{align} \label{eq:gexp}
        (D^{\ell} + 1)G_{k+1} &= \sum_{\abs{\vec{a}} + \ell(\vec{a}) \leq k+2\ell + 1} \sum_{s=0}^{t-1} b_{\vec{a}}(s)\cU_{\vec{a}}(\zeta_t^sq).
    \end{align}
    Thus, combining \eqref{eq:gexp} and \eqref{eq:grtexp}, we have
    \begin{align*}
        (D^{\ell} + 1)G_{k+1} - (D^k + 1)G_{\ell+1}^{r,t} &= \sum_{\abs{\vec{a}} + \ell(\vec{a}) \leq k+2\ell + 1} \sum_{s=0}^{t-1} b_{\vec{a}}'(s)\cU_{\vec{a}}(\zeta_t^sq) - \sum_{\abs{\vec{a}} + \ell(\vec{a}) \leq 2k+\ell + 1} \sum_{s=0}^{t-1} b_{\vec{a}}'(s)\cU_{\vec{a}}(\zeta_t^sq).
    \end{align*}
    Since $\ell > k$ and the length $\ell(\vec{a}) \geq 1$, setting 
    \begin{equation*}
        c_{\vec{a}}(s) :=
        \begin{cases}
            b_{\vec{a}}(s) - b_{\vec{a}}'(s) & \text{if } \abs{\vec{a}} + \ell(\vec{a}) \leq k + 2\ell + 1, \\
            0 & \text{otherwise},
        \end{cases}
    \end{equation*}
    for each $\vec{a}$ and $0 \leq s < t$, we may write
    \begin{align*}
        (D^{\ell} + 1)G_{k+1} - (D^k + 1)G_{\ell+1}^{r,t} &= \sum_{\abs{\vec{a}} \leq k + 2\ell} \sum_{s=0}^{t-1} c_{\vec{a}}(s)\cU_{\vec{a}}(\zeta_t^sq),
    \end{align*}
    proving our result.
\end{proof}

\subsection{Proof of Theorem \ref{thm:primes-ap}}
Let $\ell \geq 3$ be odd and $d = 2\ell + 1$. By Lemma \ref{lem:primes-ap}, $f_{1,\ell}^{r,t}$ strongly detects primes congruent to $r$ modulo $t$. The set $\{f_{1,\ell}^{r,t}~\colon \ell \geq 3\}$ thus contains infinitely many linearly independent $q$-series which strongly detect primes congruent to $r$ modulo $t$ by Lemma \ref{lem:inf-ap}, all of which may be written as
\begin{align*}
    f_{1,\ell}^{r,t} = \sum_{\abs{\vec{a}} \leq d} \sum_{s=0}^{t-1} c_{\vec{a}}(s)\cU_{\vec{a}}(\zeta_t^sq)
\end{align*}
for constants $c_{\vec{a}}(s) \in \C$ by Lemma \ref{lem:main-ap}. This completes the proof of the theorem.

\section*{Appendix}

Here we give complete expressions for the examples given in the introduction. We have that
\begin{tiny}
\begin{align*}
    g^\ast(n) &:= 111800296700031174473803912086849297415061538120147327369024000 M_{(1)}(n) \\
    &-166752101806011768239059984406622080125739818487572078291361760 M_{(3)}(n) \\
    &+62767494936926484419500539507127762695263456351080655000628648 M_{(5)}(n) \\
    &-8093327520713830454881403544414940189352187143324915609260800 M_{(7)}(n) \\
    &+280441876128809798819406207780253229131515792004956171930233 M_{(9)}(n) \\
    &-2796626716231217376090778794330046298729981595955539822240 M_{(11)}(n) \\
    &-7786923725330582178016165199689395032162976346176403887 M_{(13)}(n) \\
    &+228492332183970584572974559805984237526583822048824920 M_{(15)}(n) \\
    &-1215830626333999290688149213865467730310825113076517 M_{(17)}(n) \\
    &+2586402206320506298967818494987523980980061899880 M_{(19)}(n) \\
    &-1410067106844346141699284499808017619922382477 M_{(21)}(n) \\
    &-3994171532397060636616394065429339974329441343990850859754624000 M_{(1,1)}(n) \\
    &+2882656654596970644128099950941854712719970786862658002184240640 (M_{(1,3)}(n) + M_{(3,1)}(n)) \\
    &-55606741605506434789774529597134283862175560363960917017922880 (M_{(1,5)}(n) + M_{(5,1)}(n)) \\
    &-31553740779586632034902588733558781970650743391743394004783040 (M_{(1,7)}(n) + M_{(7,1)}(n)) \\
    &+1236365123530668916047274374112405703022697561233767698983840 (M_{(1,9)}(n) + M_{(9,1)}(n)) \\
    &-8134808611384895119269896023741057784440125881909492385720 (M_{(1,11)}(n) + M_{(11,1)}(n)) \\
    &-59580798330731620522824089676425339773368687751394331040 (M_{(1,13)}(n) + M_{(13,1)}(n)) \\
    &+726670078684395930235260971431002859974544415177094000 (M_{(1,15)}(n) + M_{(15,1)}(n)) \\
    &-2518073299368806196400137760585407383384131378897920 (M_{(1,17)}(n) + M_{(17,1)}(n)) \\
    &+1015449367245264541365106760948433593634244514120 (M_{(1,19)}(n) + M_{(19,1)}(n)) \\
    &+34425095416505212317885657088677311603740090774622195871398400000 M_{(1,1,1)}(n) \\
    &+832023026419490860720168658116504249351153130277635758982553600 (M_{(1,1,3)}(n) + M_{(1,3,1)}(n) + M_{(3,1,1)}(n)) \\
    &-543579483064233292326674999312725986268529256505271978016384000 (M_{(1,1,5)}(n) + M_{(1,5,1)}(n) + M_{(5,1,1)}(n)) \\
    &-92019503837778700907398744960854625745491505616944777544192000 (M_{(1,1,7)}(n) + M_{(1,7,1)}(n) + M_{(7,1,1)}(n)) \\
    &+4269209011502676546086277141098184955545823619844603765888000 (M_{(1,1,9)}(n) + M_{(1,9,1)}(n) + M_{(9,1,1)}(n)) \\
    &-1229617491336773807894731990558973652669489467404488268800 (M_{(1,1,11)}(n) + M_{(1,11,1)}(n) + M_{(11,1,1)}(n)) \\
    &-229089308125621772523737067819807506681260603464966144000 (M_{(1,1,13)} + M_{(1,13,1)}(n) + M_{(13,1,1)}(n)) \\
    &+2160624056705401494233143285716822486526356010777267200 (M_{(1,1,15)}(n) + M_{(1,15,1)}(n) + M_{(15,1,1)}(n)),
\end{align*}
\end{tiny}
which corresponds to the form $g_{1,3}$. We also have that
\begin{align*}
    f^\ast(n) &= 4(199 + 21\omega^{n-1} + 21\omega^{2n-2})M_{(1)}(n) + 7(1 - 80\omega^{n-1} - 80\omega^{2n-2})M_{(3)}(n)\\
    &+ 42(7 - 2\omega^{n-1} - 2\omega^{2n-2})M_{(5)}(n) + 23M_{(7)}(n)  \\
    &- 1680(11 - 4\omega^{n-1} - 4\omega^{2n-2})M_{(1,3)}(n) - 1680(11 - 4\omega^{n-1} - 4\omega^{2n-2})M_{(3,1)}(n)\\
    &- 10752M_{(1,1)}(n) - 3024M_{(1,5)}(n) - 3024M_{(5,1)}(n) +282240M_{(1,1,1)}(n) \\
    &+ 120960M_{(1,1,3)}(n) + 120960M_{(1,3,1)}(n) + 120960M_{(3,1,1)}(n) - 161280M_{(1,1,1,1)}(n),
\end{align*}
which corresponds to the form $f_{1,3}^{1,3}$.

\end{document}